\documentclass[12pt,reqno]{amsart}
\usepackage{amsmath, amssymb, amsthm}
\usepackage{graphicx, color}
\usepackage{enumerate}

\begin{document}

 \bibliographystyle{plain}
 \newtheorem{theorem}{Theorem}
 \newtheorem{lemma}[theorem]{Lemma}
 \newtheorem{corollary}[theorem]{Corollary}
 \newtheorem{problem}[theorem]{Problem}
 \newtheorem{conjecture}[theorem]{Conjecture}
 \newtheorem{definition}[theorem]{Definition}
 \newtheorem{prop}[theorem]{Proposition}
 \numberwithin{equation}{section}
 \numberwithin{theorem}{section}

 \newcommand{\mo}{~\mathrm{mod}~}
 \newcommand{\mc}{\mathcal}
 \newcommand{\rar}{\rightarrow}
 \newcommand{\Rar}{\Rightarrow}
 \newcommand{\lar}{\leftarrow}
 \newcommand{\lrar}{\leftrightarrow}
 \newcommand{\Lrar}{\Leftrightarrow}
 \newcommand{\zpz}{\mathbb{Z}/p\mathbb{Z}}
 \newcommand{\mbb}{\mathbb}
 \newcommand{\B}{\mc{B}}
 \newcommand{\cc}{\mc{C}}
 \newcommand{\D}{\mc{D}}
 \newcommand{\E}{\mc{E}}
 \newcommand{\F}{\mathbb{F}}
 \newcommand{\G}{\mc{G}}
  \newcommand{\ZG}{\Z (G)}
 \newcommand{\FN}{\F_n}
 \newcommand{\I}{\mc{I}}
 \newcommand{\J}{\mc{J}}
 \newcommand{\M}{\mc{M}}
 \newcommand{\nn}{\mc{N}}
 \newcommand{\cQ}{\mc{Q}}
 \newcommand{\cP}{\mc{P}}
 \newcommand{\U}{\mc{U}}
 \newcommand{\X}{\mc{X}}
 \newcommand{\Y}{\mc{Y}}
 \newcommand{\itQ}{\mc{Q}}
 \newcommand{\sgn}{\mathrm{sgn}}
 \newcommand{\C}{\mathbb{C}}
 \newcommand{\R}{\mathbb{R}}
 \newcommand{\T}{\mathbb{T}}
 \newcommand{\N}{\mathbb{N}}
 \newcommand{\Q}{\mathbb{Q}}
 \newcommand{\Z}{\mathbb{Z}}
 \newcommand{\A}{\mathbb{A}}
 \newcommand{\ff}{\mathfrak F}
 \newcommand{\fb}{f_{\beta}}
 \newcommand{\fg}{f_{\gamma}}
 \newcommand{\gb}{g_{\beta}}
 \newcommand{\vphi}{\varphi}
 \newcommand{\whXq}{\widehat{X}_q(0)}
 \newcommand{\Xnn}{g_{n,N}}
 \newcommand{\lf}{\left\lfloor}
 \newcommand{\rf}{\right\rfloor}
 \newcommand{\lQx}{L_Q(x)}
 \newcommand{\lQQ}{\frac{\lQx}{Q}}
 \newcommand{\rQx}{R_Q(x)}
 \newcommand{\rQQ}{\frac{\rQx}{Q}}
 \newcommand{\elQ}{\ell_Q(\alpha )}
 \newcommand{\oa}{\overline{a}}
 \newcommand{\oI}{\overline{I}}
 \newcommand{\dx}{\text{\rm d}x}
 \newcommand{\dy}{\text{\rm d}y}
\newcommand{\cal}[1]{\mathcal{#1}}
\newcommand{\cH}{{\cal H}}
\newcommand{\diam}{\operatorname{diam}}
\newcommand{\bx}{\mathbf{x}}
\newcommand{\Ps}{\varphi}

\newcommand{\zp}{\mathbb{Z}_{\textit{p}}}
\newcommand{\qp}{\mathbb{Q}_{\textit{p}}}
\newcommand{\zq}{\mathbb{Z}_{\textit{q}}}
\newcommand{\bA}{\mathbb{A}}
\newcommand{\bN}{\mathbb{N}}
\newcommand{\bQ}{\mathbb{Q}}
\newcommand{\bR}{\mathbb{R}}
\newcommand{\bZ}{\mathbb{Z}}

\parskip=0.5ex

\title[Bounded remainder sets for adelic tori]{Bounded remainder sets for rotations\\ on higher dimensional adelic tori}
\author{Akshat~Das, Joanna~Furno, Alan~Haynes}


\allowdisplaybreaks

\begin{abstract}
In this paper we give a simple, explicit construction of polytopal bounded remainder sets of all possible volumes, for any irrational rotation on the $d$ dimensional adelic torus $\A^d/\Q^d$. Our construction involves ideas from dynamical systems and harmonic analysis on the adeles, as well as a geometric argument that reduces the existence argument to the case of an irrational rotation on the torus $\R^d/\Q^d$.
\end{abstract}

\maketitle

\section{Introduction}\label{sec:intro}

Let $G$ be a compact, metrizable, Abelian group, written additively. There is a unique Haar probability measure on $G$ and, for each $\alpha\in G$, the measure-preserving map $T_\alpha: G\rar\ G$ defined by $T_\alpha(x)=x+\alpha$ is referred to as rotation by $\alpha$ on $G$. Bounded remainder sets (BRS's) for $T_\alpha$ are measurable sets $A$ with the property that there exists a constant $C=C(A)$ such that, for almost every $x\in G$ (with respect to Haar measure) and for any $N\in\N$,
\begin{equation*}
\left|\sum_{n=0}^{N-1}\chi_A(x+n\alpha)-N|A|\right|\le C,
\end{equation*}
where $|A|$ denotes the Haar measure of $A$. BRS's are allowed to be multisets, equivalently, $\chi_A$ is allowed to be a finite sum of indicator functions of measurable sets.

The study of BRS's for rotations on the groups $\T^d=\R^d/\Z^d, d\in\N,$ has a history of nearly 100 years. For $d=1$ the first results were obtained by Hecke \cite{Heck1922}, Ostrowski \cite{Ostr1927/30}, and Kesten \cite{Kest1966/67}, who proved that, for an irrational rotation by $\alpha$ on the circle $\R/\Z$, an interval $A$ will be a BRS if and only if
\begin{equation*}
|A|\in\alpha\Z+\Z.
\end{equation*}
In the $d=2$ case, Sz\"{u}sz \cite{Szus1954} constructed infinite families of BRS parallelograms. The $d\ge 2$ case was subsequently studied by Liardet \cite{Liar1987}, Rauzy \cite{Rauz1972}, Ferenczi \cite{Fere1992}, Oren \cite{Oren1982}, and Zhuravlev \cite{Zhur2005,Zhur2011,Zhur2012}, as well as other authors (e.g. \cite{HaynKoiv2016}). It follows from work of Furstenberg, Keynes, and Shapiro \cite{FursKeynShap1973}, and Hal\'asz \cite{Hala1976} that 
for any $d$ and for any $\alpha$ for which $T_\alpha$ is ergodic, the set of all volumes of BRS's for $\alpha$ is
\begin{equation*}
\{n\cdot\alpha +m\ge 0:n\in\Z^d,m\in\Z\}.
\end{equation*}
Recent work of Grepstad and Lev \cite{GrepLev2015} provides examples of \textit{parallelotope} BRS's of all possible volumes. After Grepstad and Lev's results, it was observed that the construction of BRS's for toral rotations is closely related to the construction of mathematical quasicrystals with rigid deformation properties (see \cite{DuneOgue90}). This led to another simple geometric explanation for why the parallelotopes in Grepstad and Lev's work are in fact BRS's \cite{HaynKellKoiv2017}.

Subsequent to the results described in the previous paragraph, in \cite{FurnHaynKoiv2019} attention was turned to the problem of constructing BRS's for rotations on compact subgroups of the adelic torus $\A/\Q$. The results of Hal\'asz mentioned in the previous paragraph also apply in this more general setting to give a description of the set of all volumes of BRS's. Therefore the emphasis is on actually constructing such sets in a geometrically appealing way. The proofs given in \cite{FurnHaynKoiv2019} involved new constructions using $p$-adic cut and project sets, and they are significant in that they extend previous results on the real torus to uncountably many topological group isomorphism classes of connected, compact, metrizable, Abelian groups. The goal of this paper is to complete this direction of inquiry by giving a simple geometric construction of BRS's of all possible volumes, for all ergodic rotations on adelic tori $\A^d/\Q^d$ in dimensions $d\ge 1$.

In order to state our main result, let $\A$ denote the ring of adeles over $\Q$ with the usual restricted product topology, and let $\mc{P}$ denote the collection of all prime numbers (more detailed definitions are given in the next section). Suppose $\mc{Q}$ is a (finite or infinite) subset of $\mc{P}$, write $\mc{Q}=\{p_1,p_2,\ldots\}$, and let $\A_\mc{Q}$ be the projection of $\A$ onto the coordinates indexed by the infinite place and the elements of $\mc{Q}$. We take $\A_\mc{Q}$ with the natural restricted product topology, which is also the final topology with respect to this projection. The additive group $\Gamma_\mc{Q}=\Z[1/p_1,1/p_2,\ldots ]$ can be embedded diagonally in $\A_\mc{Q}$ via the map $\gamma\mapsto (\gamma,\gamma,\ldots )$, and we identify $\Gamma_\mc{Q}$ with its image under this embedding, which is a discrete subgroup of $\A_\mc{Q}$. The quotient group
\begin{equation*}
X_\mc{Q}=\A_\mc{Q}/\Gamma_\mc{Q},
\end{equation*}
is easily seen to be a connected, compact, metrizable, Abelian group.

For $d\in\N$, we are interested in bounded remainder sets for rotations on $X_{\cQ}^d \cong \bA_{\cQ}^d / \Gamma_{\cQ}^d$ (taken with the product topology). Since it is convenient to examine all the coordinates at each place at once, we consider $\bA_{\cQ}^d$ as a subset of $\bR^d \times \prod_{p \in \cQ}\qp^d$. Thus, we consider elements of the form $\vec{\alpha}= ( \vec{\alpha}_{\infty}, \vec{\alpha}_{p_1}, \vec{\alpha}_{p_2} \ldots)$, where
\begin{align*}
\vec{\alpha}_{\infty} &= (\alpha_{\infty, 1}, \alpha_{\infty,2},\ldots,\alpha_{\infty,d}),~\text{and}\\
\vec{\alpha}_p &= ( \alpha_{p,1}, \alpha_{p,2},\ldots,\alpha_{p,d})~\text{for}~p \in \cQ.
\end{align*}
Rotation by $\vec{\alpha}$ on $X_\mc{Q}^d$ is the map $T_{\vec{\alpha}}: X_{\cQ}^d \rightarrow X_{\cQ}^d$ defined by $T_{\vec{\alpha}}(\vec{x}) = \vec{x} + \vec{\alpha}$. Our main theorem provides a construction of adelic polytope BRS's for $T_{\vec{\alpha}}$, of all possible volumes, in the generic case when this map is ergodic.
\begin{theorem}\label{thm:main}
Suppose $\cQ \subseteq \cP$ and that $X_{\cQ}^d$ is defined as above. Suppose that $\vec{\alpha} \in X_{\cQ}^d$ and that $1, \alpha_{\infty, 1},\ldots,\alpha_{\infty,d}$ are linearly independent over $\bQ$. Then the set of all volumes of BRS's for $T_{\vec{\alpha}}$ is 
\begin{equation}\label{eq:volumes}
	\left\{ \sum_{j=1}^d \left(\gamma_j\alpha_{\infty,j} - \sum_{p \in \cQ}\left\{\gamma_j\alpha_{p,j}\right\}_{p}\right) + \eta \geq 0 : \gamma_j \in \Gamma_{\cQ}, \eta \in \bZ \right\},
\end{equation}
where $\left\{ \cdot\right\}_p: \bQ_p \rightarrow \bR$ is the $p$-adic fractional part. Furthermore, for every volume in this set, there is a BRS for $T_{\vec{\alpha}}$ of that volume which is the projection to $X_{\cQ}^d$ of the Cartesian product of a parallelotope in $\R^d$ with balls centered at $0$ in the $p$-adic directions (all but finitely many of which have radius 1).
\end{theorem}

The $p$-adic fractional part is defined in the next section. The condition that the numbers $1, \alpha_{\infty, 1},\ldots,\alpha_{\infty,d}$ are linearly independent over $\bQ$ is necessary and sufficient to ensure that the map $T_{\vec{\alpha}}$ is ergodic. Furthermore, ergodicity in this setting is equivalent to unique ergodicity, and also to the requirement that every orbit of $T_{\vec{\alpha}}$ is dense in $X_{\cQ}^d$. See Lemma \ref{lem:UErg} below for references and proofs of these statements.

The layout of this paper is as follows: In Section \ref{sec:background} we present background material and notation. In Section \ref{sec:restriction} we show that, for $\vec{\alpha}$ satisfying the hypothesis of our main theorem, the volume of any BRS for $T_{\vec{\alpha}}$ must belong to the set \eqref{eq:volumes}. In Section \ref{sec:construction} we construct BRS's for each of these allowable volumes, by scaling the sets and translations in order to shift the focus to the Archimedean coordinates.

\section{Background}\label{sec:background}

For background material on $p$-adic analysis and Fourier analysis over local fields we refer the reader to \cite[Chapter 1]{Kobl1984} and \cite[Chapter 3]{RamaVale1999}.

Let $p$ be a prime number, let $\Q_p$ denote the field of $p$-adic numbers, and let $|\cdot|_p$ denote the usual $p$-adic absolute value on this field. The ring of $p$-adic integers $\Z_p$ is the set of $x \in \Q_p$ with $|x|_p \leq 1$. Every element $x \in \Q_p$ can be expressed as a sum of the form
\begin{equation*}
x = \sum_{i = N}^{\infty} x_i p^i,
\end{equation*}
where $x_i \in \left\{ 0, 1, \ldots, p-1 \right\}$ for all integers $i \geq N$. We define $p$-adic fractional part and the $p$-adic integer part of such an element by
\begin{equation*}
\left\{x \right\}_p = \sum_{i=N}^{-1} x_i p^i \quad\text{ and }\quad \left\lfloor x\right\rfloor_p = \sum_{i = 0}^{\infty} x_i p^i,
\end{equation*}
respectively, with the usual convention that the empty sum is 0. To parallel the $p$-adic notation, we will use $|\cdot|_{\infty}$, $\left\{\cdot \right\}_{\infty}$ , and $\left\lfloor\cdot \right\rfloor_{\infty}$ to denote the usual Archimedean absolute value, fractional part, and integer part on $\R$.

The field $\Q_p$ is locally compact and its additive group has Pontryagin dual $\widehat{\Q}_p \cong \Q_p$. To make this isomorphism explicit, let
$e(z) = e^{2\pi iz}$ and, for $y \in \Q_p$, define $\psi_y\in\widehat{\Q}_p$ by 
\begin{eqnarray*}
\psi_y : \Q_p &\rightarrow & \C\\
x & \mapsto & e(\left\{yx \right\}_p).
\end{eqnarray*}
The map $y \mapsto \psi_y$ is then a topological group isomorphism between $\Q_p$ and $\widehat{\Q}_p$.

Let $\cP$ be the set of all prime numbers, and let 
$\cQ = \left\{p_1, p_2, \ldots \right\}$ be a nonempty subset of $\cP$. As a set, $\A_\mc{Q}$ is defined to be the collection of elements
\begin{equation*}
\alpha=(\alpha_\infty,\alpha_{p_1},\alpha_{p_2},\ldots)\in\R\times\prod_{p\in\mc{Q}}\Q_p
\end{equation*}
which satisfy $\alpha_p\in\Z_p$ for all but finitely many primes $p\in\mc{Q}$. It follows from the strong triangle inequality for the non-Archimedean absolute values that the elements of $\A_\mc{Q}$ form a ring under pointwise addition and multiplication. A natural topology on $\A_\mc{Q}$ is the restricted product topology, with respect to the sets $\Z_p$, for $p\in\mc{Q}$. With this topology the additive group of $\A_\mc{Q}$ is a locally compact topological group. This group is also self dual, with an explicit isomorphism given by the map from $\A_\mc{Q}$ to $\widehat{\A}_\mc{Q}$ defined by $\beta\mapsto\psi_\beta^\mc{Q},$ where
\begin{equation*}
\psi_\beta^\mc{Q}(\alpha)=e(\beta_\infty\alpha_\infty)\cdot\prod_{p\in\mc{Q}}e(-\{\beta_p\alpha_p\}_p).
\end{equation*}
Note that in this product, all but finitely many terms are equal to $1$.

Now let $d$ be a positive integer, and consider  $\A_{\cQ}^d$ with the product topology, and with elements $\vec{\alpha}\in \A_{\cQ}^d$ denoted as in the introduction. The group $\Gamma_\mc{Q}=\Z[1/p_1,1/p_2,\ldots]$ embeds in $\A_\mc{Q}$ by the map $\gamma\mapsto (\gamma,\gamma,\gamma\ldots)$, and we identify it with its image under this map. With this identification, it is easy to check that $\Gamma_\mc{Q}$ is a discrete and closed subgroup of $\A_\mc{Q}$, and that a strict fundamental domain for the quotient group $X_\mc{Q}=\A_\mc{Q}/\Gamma_\mc{Q}$ is given by the collection of points
\begin{equation*}
[0,1)\times\prod_{p\in\mc{Q}}\Z_p\subseteq\A_\mc{Q}.
\end{equation*}
It follows that  $X_{\cQ}^d = \A_{\cQ}^d / \Gamma_{\cQ}^d$ is a compact group with strict fundamental domain
\[F_{\cQ}^d = [0,1)^d \times \prod_{p \in \cQ}\zp^d,\]
and that its dual group is the subset of characters on $\A_\mc{Q}^d$ which are trivial on $\Gamma_\mc{Q}^d$. Explicitly, the map from the discrete group $\Gamma_\mc{Q}^d$ to $\widehat{X}_{\cQ}^d$ given by  $\vec{\gamma} = (\gamma_1, \ldots, \gamma_d) \mapsto \psi_{\vec{\gamma}}^{\cQ}$, where
\begin{eqnarray*}
\psi_{\vec{\gamma}}^{\cQ} : X_{\cQ}^d &\rightarrow & \C\\
\vec{\alpha} & \mapsto & \prod_{j=1}^d \left(e(\gamma_j \alpha_{\infty,j})\prod_{p \in \cQ}e(-\left\{\gamma_j \alpha_{p,j} \right\}_{p})\right),
\end{eqnarray*}
is a topological group isomorphism.

If the translation $T_{\vec{\alpha}}$ is uniquely ergodic, then the convergence of the Birkhoff sums to their ergodic averages is independent of the starting point $\vec{x} \in X_{\cQ}^d$. This is important in what follows, as it implies that all translates of BRS's for 
$T_{\vec{\alpha}}$ are also BRS's for $T_{\vec{\alpha}}$. Recall the following general form of Weyl's criterion.
\begin{lemma}\cite[Chapter 4, Corollary 1.2]{KuipNied1974}\label{lem:Weyl}
  Suppose that $G$ is a compact Abelian group. A sequence $\{x_n\}_{n\in\N}\subseteq G$ is uniformly distributed in $G$ with respect to Haar measure if and only if, for every nontrivial character $\chi\in\widehat{G}$,
  \begin{equation*}
    \lim_{N\rar\infty}\frac{1}{N}\sum_{n=1}^N\chi(x_n)=0.
  \end{equation*}
\end{lemma}
We now demonstrate how Lemma \ref{lem:Weyl} can be used to easily classify the collection of all $\vec{\alpha}\in X_\mc{Q}^d$ for which $T_{\vec{\alpha}}$ is uniquely ergodic.

\begin{lemma}\label{lem:UErg}
	The translation $T_{\vec{\alpha}}$ is uniquely ergodic if and only if it is ergodic, and it is ergodic if and only if the real numbers
	\[1, \alpha_{\infty,1}, \alpha_{\infty,2}, \ldots, \alpha_{\infty,d}\] are linearly independent over $\Q$.
\end{lemma}

\begin{proof} First of all, it is clear that the linear independence over $\Q$ of the real numbers in the statement of the lemma does not depend on the choice of representative for $\vec{\alpha}$ in $X_\mc{Q}^d.$
For rotations on compact metrizable groups, the existence of a dense orbit is equivalent to unique ergodicity of the rotation, with Haar measure as the unique ergodic measure (see \cite[Theorem 4.14]{EinsWard2011}). In particular, for a fixed rotation, every orbit is dense if and only if the orbit of 0 is dense. Thus, we focus on whether or not $\left\{ n \vec{\alpha}\right\}_{n\in\N}$ is dense in $X_{\cQ}^d$.

One direction is easy and does not require Weyl's criterion. If 
$1, \alpha_{\infty,1}$, $\alpha_{\infty,2}, \ldots, \alpha_{\infty,d}$ are linearly dependent over $\Q$, then $\left\{ n \vec{\alpha}\right\}_{n\in\N}$ is not dense
in $X_{\cQ}^d$, which implies that $T_{\vec{\alpha}}$ is not uniquely ergodic.

For the converse, suppose that $1, \alpha_{\infty,1}, \alpha_{\infty,2}, \ldots, \alpha_{\infty,d}$ are linearly independent over $\Q$. Let $\vec{\gamma} = (\gamma_1, \ldots, \gamma_d) $ be a non-zero element of $\Gamma_{\cQ}^d$. Then
\begin{equation*}
\theta = \sum_{j=1}^d \left(\gamma_j\alpha_{\infty,j} - \sum_{p \in \cQ}\left\{\gamma_j\alpha_{p,j}\right\}_{p}\right)
\end{equation*}
is an irrational number. (Recall that the inner sum in the definition of $\theta$ has only finitely many nonzero terms.) For any prime $p$, any integer $j$ such that $1 \leq j \leq d$, and any $n \in \Z$,
\begin{equation*}
 n \left\{\gamma_j \alpha_{p,j}\right\}_p -\left\{n \gamma_j \alpha_{p,j} \right\}_p\in \Z.
\end{equation*}
Thus,
\begin{equation*}
\frac{1}{N} \sum_{n=1}^N \psi_{\vec{\gamma}}^{\cQ}(n \vec{\alpha}) = \frac{1}{N} \sum_{n=1}^N e(n\theta),
\end{equation*}
 which tends to $0$ as $N \rightarrow \infty$ by the irrationality of $\theta$. By Lemma \ref{lem:Weyl}, this implies that $\left\{ n \vec{\alpha}\right\}_{n\in\N}$ is dense in $X_{\cQ}^d$, so $T_{\vec{\alpha}}$ is uniquely ergodic.
\end{proof}

\section{Restriction of possible volumes of BRS's}\label{sec:restriction}

In this section we prove that, if $\vec{\alpha}$ satisfies the hypothesis of Theorem \ref{thm:main}, then the volume of any BRS for $T_{\vec{\alpha}}$ must be an element of the set in \eqref{eq:volumes}. The argument for $X_{\cQ}^d$ 
is analogous to the argument for $X_{\cQ}$ in \cite{FurnHaynKoiv2019}. Moreover, similar arguments appear in many places, such as \cite{FursKeynShap1973}, \cite[Theorem 2]{Hala1976}, \cite[Section 2.2]{GrepLev2015}, and \cite[Chapter 14]{GottHedl1955}, so the underlying ideas in this section are not new.  We include these arguments only for completeness, since the Fourier analysis in the proof is taking place in a slightly different setting than usual. The following lemma relates the existence of BRS's to the existence of dynamical coboundaries for the map $T_{\vec{\alpha}}$.

\begin{lemma}\label{lem:CoBound}
  A measurable set $A\subseteq X_\cQ^d$ is a BRS for $T_{\vec{\alpha}}$ if and only if there exists a bounded, measurable function $g:X_\cQ^d\rar\R$ satisfying
  \begin{equation}\label{eq:CoBoundEqn}
    \chi_A(\vec{x})-|A|=g(\vec{x})-g(\vec{x}+\vec{\alpha}),
  \end{equation}
  for all $\vec{x}\in X_\cQ^d$.
\end{lemma}

\begin{proof}
Suppose that $A$ is a measurable subset of $X_{\cQ}^d$ and that there exists a bounded, measurable function $g:X_\cQ^d\rar\R$ satisfying \eqref{eq:CoBoundEqn} for all $\vec{x}\in X_\cQ^d$.
Then we have that
\begin{equation*}
  \left|\sum_{n=0}^{N-1}\chi_A(\vec{x}+n\vec{\alpha})-N|A|\right|=\left|g(\vec{x})-g(\vec{x}+N\vec{\alpha})\right|\le 2 \|g\|_\infty,
\end{equation*}
where $\|g\|_\infty = \sup_{\vec{x} \in X_{\cQ}^d} |g(\vec{x})|$. Hence, $A$ is a BRS for $T_{\vec{\alpha}}$.

Conversely, suppose that $A\subseteq X_{\cQ}^d$ is measurable and a BRS for $T_{\vec{\alpha}}$. Then the function $g:X_\cQ^d\rar\R$ defined by
\begin{equation*}
g(\vec{x})=\liminf_{N\rar\infty}\left(\sum_{n=0}^{N-1}\chi_A(\vec{x}+n\vec{\alpha})-N|A|\right)
\end{equation*}
is well-defined, bounded, and measurable (by Fatou's lemma). By rearranging sums, we can write
\begin{align*}
 &\chi_A(\vec{x})-|A|+\sum_{n=0}^{N-1}\left(\chi_A(\vec{x}+(n+1)\vec{\alpha})-|A|\right)\\
 &\hspace*{1in}=\sum_{n=0}^N\left(\chi_A(\vec{x}+n\vec{\alpha})-|A|\right).
\end{align*}
Taking the limit inferior of both sides, we see that \eqref{eq:CoBoundEqn} holds.
\end{proof}
Next we use Lemma \ref{lem:CoBound} together with a Fourier analysis argument to prove the following result.
\begin{prop} \label{prop:volumes}
Suppose that $\vec{\alpha} \in X_{\cQ}^d$ and that the numbers
\[1, \alpha_{\infty,1}, \alpha_{\infty,2}, \ldots, \alpha_{\infty,d}\] are linearly independent over $\Q$. If a measurable set $A \subseteq X_{\cQ}^d$ is
a BRS for $T_{\vec{\alpha}}$, then there exist $\vec{\gamma} \in \Gamma_{\cQ}^d$ and $\eta \in \Z$ such that
\begin{equation}
	|A| = \sum_{j=1}^d \left(\gamma_j\alpha_{\infty,j} - \sum_{p \in \cQ}\left\{\gamma_j\alpha_{p,j}\right\}_{p}\right) + \eta.
\label{eq:volume}
\end{equation}
\end{prop}

\begin{proof}
Suppose that $A\subseteq X_{\cQ}^d$ is a measurable set that is a BRS for $T_{\vec{\alpha}}$. By Lemma \ref{lem:CoBound}, there exists a bounded, measurable function $g:X_\cQ^d\rar\R$ satisfying equation (\ref{eq:CoBoundEqn}) for all $\vec{x}\in X_\cQ^d$. Define
\begin{eqnarray*}
	h: X_{\cQ}^d &\rightarrow& \C \\
	\vec{x} &\mapsto & e(g(\vec{x})).
\end{eqnarray*}
It follows that
\begin{equation*}
h(\vec{x} + \vec{\alpha}) = e(g(\vec{x}) - \chi_A(\vec{x}) + |A|)
 = e(|A|)h(\vec{x}).
\end{equation*}
Let $\widehat{h}(\vec{\gamma}),~\vec{\gamma} \in \Gamma_{\cQ}^d,$ denote the Fourier coefficients of $h$. Then
\begin{align*}
	0 & =e(|A|)h(\vec{x}) - h(\vec{x} + \vec{\alpha})\\
	& = \sum_{\vec{\gamma} \in \Gamma_{\cQ}^d} \left(e(|A|)-\psi_{\vec{\gamma}}^{\cQ}(\vec{\alpha})\right)\widehat{h}(\vec{\gamma})\psi_{\vec{\gamma}}^{\cQ}(\vec{x}).
\end{align*}
Since $h(\vec{x})$ is not identically zero, Plancherel's theorem implies that some coefficient is zero, so 
$e(|A|) = \psi_{\vec{\gamma}}^{\cQ}(\vec{\alpha})$ for some $\vec{\gamma} \in \Gamma_{\cQ}$. By the definition of $\psi_{\vec{\gamma}}^{\cQ}$, there exists an $\eta \in \Z$ such that (\ref{eq:volume}) holds.
\end{proof}

\section{Construction of BRS's of all allowable volumes}\label{sec:construction}

 Fix $\vec{\alpha}\in X_{\cQ}^d$, and suppose that $1, \alpha_{\infty,1},\alpha_{\infty,2},\ldots,\alpha_{\infty,d} \in \bR$ are linearly independent over $\bQ$. Define $T_{\vec{\alpha}}: X_{\cQ}^d \rightarrow X_{\cQ}^d$ by $T_{\vec{\alpha}}(\vec{x}) = \vec{x} + \vec{\alpha}$. Suppose $\gamma_1, \gamma_2,\ldots, \gamma_d \in \Gamma_{\cQ} \backslash \left\{0\right\}$ and $\eta \in \bZ$ are chosen such that 
\begin{equation}\label{eqn:VolFormula}
	V = \sum_{j=1}^d \left(\gamma_j\alpha_{\infty,j} - \sum_{p \in \cQ}\left\{\gamma_j\alpha_{p,j}\right\}_{p}\right) + \eta >0.
\end{equation} 
We will explain how to deal with the case when some of the $\gamma_i$ are 0 at the end of the proof. For $1\leq j \leq d$, let $ \delta_j$ be the denominator of $\gamma_j$, that is, the product of all $|\gamma_j|_p$ such that $|\gamma_j|_p>1$. Then there exist $g_j \in \bZ$ such that $\gamma_j = \dfrac{g_j}{\delta_j}$. Let $D$ be the $d\times d$ diagonal matrix with $\delta_1, \delta_2, \ldots, \delta_d$ on the diagonal. Since $\delta_1, \delta_2, \ldots, \delta_d \in \Gamma_{\cQ} \backslash \left\{0\right\}$, the matrix 
$D$ is invertible. Moreover, the diagonal entries of $D^{-1}$ are also in $\Gamma_{\cQ} \backslash \left\{0\right\}$. As with elements of $\Gamma_{\cQ}$, we will abuse notation and consider $D$ and $D^{-1}$ as acting on $\bR^d$, $\bQ_p^d$, $\bA_{\cQ}^d$, and $X_{\cQ}^d$, interpreted appropriately in each case.
Let
\begin{equation*}
\vec{\beta} = \left( \frac{\alpha_{\infty,1}}{\delta_1}  - \sum_{p \in \cQ}\left\{\frac{\alpha_{p,1}}{\delta_1}\right\}_{p}, \ldots, \frac{\alpha_{\infty,d}}{\delta_d}  - \sum_{p \in \cQ}\left\{\frac{\alpha_{p,d}}{\delta_d}\right\}_{p} \right)\in\R^d.
\end{equation*}
Since $g_1, g_2,\ldots,g_d$ are integers, the number
\begin{equation*}
\eta' =\sum_{j=1}^{d}\sum_{p \in \cQ}\left( g_j\left\{\frac{\alpha_{p,j}}{\delta_j} \right\}_p - \left\{\frac{g_j\alpha_{p,j}}{\delta_j} \right\}_p\right)  + \eta
\end{equation*}
is also an integer. Thus, $g_1,g_2,\ldots,g_d$, and $\eta'$ are
integers satisfying the equation
\begin{equation*}
	V = \sum_{j=1}^d g_j \left( \frac{\alpha_{\infty,j}}{\delta_j}  - \sum_{p \in \cQ}\left\{\frac{\alpha_{p,j}}{\delta_j}\right\}_{p}\right) + \eta'.
\end{equation*}
Take $P_B$ to be a parallelotope in $\bR^d$ with volume $V$, spanned by vectors $\vec{v}_1, \vec{v}_2,\ldots, \vec{v}_d \in \bZ \vec{\beta} + \bZ^d$. This is possible because of \cite[Corollary 2]{GrepLev2015}, and it also follows from \cite[Theorem 1]{GrepLev2015} that $P_B$ is a BRS for the toral rotation $S_{\vec{\beta}}: \bR^d/ \bZ^d \rightarrow \bR^d/\bZ^d$ defined by 
\[S_{\vec{\beta}}(\vec{x}) = \vec{x} + \vec{\beta}.\]
Let $P_A$ be the parallelotope of volume $\delta_1\delta_2\cdots\delta_d V$ in $\R^d$ spanned by the vectors $D \vec{v}_1, D\vec{v}_2,\ldots,D\vec{v}_d$, and define sets $A$ and $B$ in $X_{\cQ}^d$ by
\begin{align*}
	A &= P_A \times  \prod_{p \in \cQ} (\delta_1 \zp \times \delta_2 \zp\times \cdots \times \delta_d \zp) \quad\text{ and }\\
	B &= P_B \times  \prod_{p \in \cQ} \bZ_p^d.
\end{align*}
For each $1 \leq j \leq d$, the set $\delta_j \zp$ is equal to $\zp$ for all but finitely many $p \in \cQ$ because $|\delta_j|_p = 1$ for all but finitely many $p \in \cQ$.
The adelic polytopes $A$ and $B$ are measurable and have measure $V$ (as multisets) with respect to normalized Haar measure on $X_{\cQ}^d$. The remainder of the proof will be devoted to proving that $A$ is a BRS for $T_{\vec{\alpha}}$.

\begin{lemma}  \label{lemma:rescale}
For all $n \in \bN$, $\chi_A(T^n_{\vec{\alpha}} (0)) = \chi_B(T^n_{D^{-1}\vec{\alpha}}(0))$.
\end{lemma}

\begin{proof}
Let $k, n \in \bN$. First note that 
\[T^n_{\vec{\alpha}}(0) = n\vec{\alpha}\quad\text{ and }\quad T^n_{D^{-1}\vec{\alpha}}(0) = nD^{-1}\vec{\alpha}.\]
If $\chi_A(T^n_{\vec{\alpha}} (0)) = k$, then there exist exactly $k$ vectors $\vec{q}_1, \ldots, \vec{q}_k \in \Gamma_{\cQ}^d$ such that $n\vec{\alpha} + \vec{q}_i \in A$ for $1 \leq i \leq k$. Moreover, $n\vec{\alpha} + \vec{q}_i \in A$ if and only if 
\[n\vec{\alpha}_{\infty}+ \vec{q}_i \in P_A\quad\text{ and }\quad n \alpha_{p,j}  + q_{i,j} \in \delta_j  \zp,\] for all $p \in \cQ$ and $1 \leq j \leq d$. We have $n\vec{\alpha}_{\infty}+ \vec{q}_i \in P_A$ if and only if there exist real numbers $t_1, t_2,\ldots,t_d \in [0,1)$ such that 
\[n\vec{\alpha}_{\infty}+ \vec{q}_i = t_1 D \vec{v}_1 + t_2 D \vec{v}_2 + \cdots + t_d D \vec{v}_d,\] which holds if and only if 
\[D^{-1}(n\vec{\alpha}_{\infty}+ \vec{q}_i) = t_1\vec{v}_1 + t_2\vec{v}_2 + \cdots + t_d \vec{v}_d \in P_B.\] Moreover, 
$n \alpha_{p,j}  + q_{i,j} \in \delta_j  \zp$ for all $p \in \cQ$ and $1 \leq j \leq d$ if and only if $(n \alpha_{p,j}  + q_{i,j})/\delta_j \in \zp$ for all $p \in \cQ$ and $1 \leq j \leq d$. Thus, the vectors in $\Gamma_{\cQ}^d$ such that $nD^{-1} \vec{\alpha} + D^{-1}\vec{q}_j \in B$ are precisely $D^{-1}\vec{q}_1, \ldots, D^{-1}\vec{q}_k$, and it follows that $\chi_B(T^n_{D^{-1}\vec{\alpha}} (0)) = k$.

The proof of the other direction, that $\chi_B(T^n_{D^{-1}\vec{\alpha}} (0)) = k$ implies that  $\chi_A(T^n_{\vec{\alpha}} (0)) = k$, follows from a similar argument.
\end{proof}

\begin{lemma}\label{lemma:getint}
For all $\lambda \in\Gamma_{\cQ}$, we have that 
\[\displaystyle \lambda - \sum_{p \in \cQ} \left\{\lambda \right\}_p \in\Z.\]
\end{lemma}

\begin{proof}
For each $q \in \cQ$ we have that
\[\left|\lambda - \sum_{p \in \cQ} \left\{\lambda \right\}_p \right|_q\le\max\left\{|\lambda - \left\{\lambda \right\}_q|_q, \max_{p\in\cQ\setminus\{q\}}\left\{|\left\{\lambda \right\}_p|_q\right\}\right\}\le 1.\] 
Therefore
\[\lambda -\sum_{p \in \cQ} \left\{\lambda \right\}_p \in \zq\] 
for all $q \in \cQ$ and the conclusion follows.
\end{proof}

Finally we have the following result, which reduces our problem to a problem on $\R^d/\Z^d$.

\begin{lemma} \label{lemma:reducetocircle}
For all $n \in \bN$, $\chi_B(T^n_{D^{-1}\vec{\alpha}}(0)) = \chi_{P_B}(S^n_{\vec{\beta}}(0))$.
\end{lemma}

\begin{proof}
Let $k, n \in \bN$. 
If $\chi_B(T^n_{D^{-1}\vec{\alpha}} (0)) = k$, then there exist exactly $k$ vectors $\vec{q}_1, \ldots, \vec{q}_k \in \Gamma_{\cQ}^d$ such that 
\[nD^{-1}\vec{\alpha}  + \vec{q}_i \in B\quad\text{ for }\quad 1 \leq i \leq k.\] Since $n\alpha_{p,j} /\delta_j  + q_{i,j} \in \zp$ for all $p \in \cQ$ and $1 \leq j \leq d$, we have $\left\{q_{i,j} \right\}_p  = -\left\{ n\alpha_{p,j}/\delta_j\right\}_p$ for all $p \in \cQ$ and $1 \leq j \leq d$. By Lemma \ref{lemma:getint}, there exist integers $n_{i,j}$ such that 
\begin{align*}
	q_{i,j} &= n_{i,j} + \sum_{p \in \cQ} \left\{ q_{i,j}\right\}_p \\
		&= n_{i,j} - \sum_{p \in \cQ}\left\{\frac{n\alpha_{p,j}}{\delta_j}  \right\}_p.
\end{align*}
Hence, there are exactly $k$ vectors $\vec{n}_i = (n_{i,1}, n_{i,2}, \ldots, n_{i,d}) \in \bZ^d$ such that the vectors
\begin{align*}
\left(\frac{n\alpha_{\infty,1}}{\delta_1} - \sum_{p \in \cQ}\left\{ \frac{n\alpha_{p,1}}{\delta_1}  \right\}_p  + n_{i,1}, \ldots,
  \frac{n\alpha_{\infty,d}}{\delta_d} - \sum_{p \in \cQ}\left\{ \frac{n\alpha_{p,d}}{\delta_d}  \right\}_p  + n_{i,d}\right)
\end{align*}
are in  the polytope $P_B$. For each $p \in \cQ$ and $1 \leq j \leq d$, let 
\[l_{p,j} =n\left\{\alpha_{p,j}/\delta_j  \right\}_p -  \left\{ n\alpha_{p,j}/\delta_j  \right\}_p,\] and note that $l_{p,j}\in\Z$. Then the vectors
\[\vec{m}_i = \left(n_{i,1} + \sum_{p \in \cQ} l_{p,1}, \ldots, n_{i,d}+\sum_{p \in \cQ} l_{p,d}\right) \in \bZ^d, \] for $1 \leq i \leq k$, are exactly the integer vectors satisfying \[n\vec{\beta} +\vec{m}_i = nD^{-1}\vec{\alpha}_{\infty} + \vec{q}_i  \in P_B.\] Hence, $\chi_{P_B}(S^n_{\vec{\beta}}(0)) = k$.

Again, the proof of the other direction, that $\chi_{P_B}(S^n_{\vec{\beta}}(0)) = k$ implies that  $\chi_B(T^n_{D^{-1}\vec{\alpha}} (0)) = k$, follows from a similar argument.
\end{proof}
In light of Lemmas \ref{lemma:rescale} and \ref{lemma:reducetocircle}, and the fact that $P_B$ is a BRS for $S_{\vec{\beta}}$, it is now clear that the set $A$ satisfies the conclusion of Theorem \ref{thm:main}.

Finally, suppose that $\gamma_1, \gamma_2,\ldots, \gamma_d \in \Gamma_{\cQ} $ and $\eta \in \bZ$ are chosen such that \eqref{eqn:VolFormula} holds, and that exactly $d_0>0$ of the $\gamma_i$ are 0. If $d_0=d$ then we can take our BRS to be $\eta$ copies of $X_\mc{Q}^d$. Otherwise, we can repeat the argument above to construct a BRS of volume $V$ in the $d-d_0$ coordinates for which $\gamma_i\not=0$, and then we can take the Cartesian product of this set with $[0,1)\times \prod_{p\in\mc{Q}}\Z_p$ in the $d_0$ coordinates where $\gamma_i=0$. The resulting set is a BRS of volume $V$ for $T_{\vec{\alpha}}$ satisfying the conclusion of Theorem \ref{thm:main}. This completes the proof of our main result.

\vspace{.15in}

{\footnotesize
	\noindent
	AD, AH: Department of Mathematics, University of Houston,\\
	Houston, TX, United States.\\
	atdas@math.uh.edu, haynes@math.uh.edu\\
	
	\noindent
	JF: Mathematical Sciences, DePaul University,\\
	Chicago, IL, United States.\\
	jfurno@depaul.edu
}


\begin{thebibliography}{1}




\bibitem{DuneOgue90}

M.~Duneau, C.~Oguey: {\em Displacive transformations and quasicrystalline symmetries}, J. Physique  51 (1990), no. 1, 5--19.

\vspace*{.05in}



\bibitem{EinsWard2011}
M.~Einsiedler, T.~Ward:
{\em Ergodic theory with a view towards number theory},
Graduate Texts in Mathematics, 259. Springer-Verlag London, Ltd., London, 2011.

\vspace*{.05in}



\bibitem{Fere1992}
S.~Ferenczi:
{\em Bounded remainder sets},
Acta Arith.  61  (1992),  no. 4, 319--326.


\vspace*{.05in}

\bibitem{FurnHaynKoiv2019}
J.~Furno, A.~Haynes, H.~Koivusalo:
{\em Bounded remainder sets for rotations on the adelic torus},
Proc. Amer. Math. Soc. 147 (2019), 5105--5115.

\vspace*{.05in}

\bibitem{FursKeynShap1973}
H.~Furstenberg, H.~Keynes, L.~Shapiro:
{\em Prime flows in topological dynamics},
 Israel J. Math.  14  (1973), 26--38.

\vspace*{.05in}

\bibitem{GottHedl1955}
W.~H.~Gottschalk, G.~A.~Hedlund:
{\em Topological dynamics},
American Mathematical Society Colloquium Publications, Vol. 36. American Mathematical Society, Providence, R. I.,  1955.

\vspace*{.05in}


\bibitem{GrepLev2015}
S.~Grepstad, N.~Lev:
{\em Sets of bounded discrepancy for multi-dimensional irrational rotation},
 Geom. Funct. Anal.  25  (2015),  no. 1, 87--133.

\vspace*{.05in}

\bibitem{Hala1976}
G.~Hal\'asz:
{\em Remarks  on  the  remainder  in  Birkhoff's  ergodic  theorem},  Acta  Math.  Acad.  Sci.
Hungar. 28 (1976), no. 3--4, 389--395.


\vspace*{.05in}

\bibitem{HaynKellKoiv2017}
A.~Haynes, M.~Kelly, H.~Koivusalo:
{\em Constructing bounded remainder sets and cut-and-project sets which are bounded distance to lattices, II},
Indag. Math. (N.S.)  28  (2017),  no. 1, 138--144.



\vspace*{.05in}

\bibitem{HaynKoiv2016}
A.~Haynes, H.~Koivusalo:
{\em Constructing bounded remainder sets and cut-and-project sets which are bounded distance to lattices},
Israel J. Math.  212  (2016),  no. 1, 189--201.



\vspace*{.05in}


\bibitem{Heck1922}
E.~Hecke: \emph{\"{U}ber analytische Funktionen und die Verteilung von Zahlen mod. eins.} (German),  Abh. Math. Sem. Univ. Hamburg  1 (1922), no. 1, 54--76.



\vspace*{.05in}

\bibitem{Liar1987}
P.~Liardet: {\em Regularities of distribution}, Compositio Math.  61  (1987),  no. 3, 267--293.


\vspace*{.05in}


\bibitem{Kest1966/67}
H.~Kesten: \emph{On a conjecture of Erd\H{o}s and Sz\"{u}sz related to uniform distribution mod $1$}, Acta Arith. 12  (1966/1967), 193--212.


\vspace*{.05in}

\bibitem{Kobl1984}
N.~Koblitz:
{\em p-adic numbers, p-adic analysis, and zeta-functions},
Second edition, Graduate Texts in Mathematics 58. Springer-Verlag, New York,  1984.

\vspace*{.05in}

\bibitem{KuipNied1974}
L.~Kuipers, H.~Niederreiter:
\emph{Uniform distribution of sequences},
Pure and Applied Mathematics,
Wiley-Interscience [John Wiley \& Sons], New York-London-Sydney,  1974.

\vspace*{.05in}

\bibitem{Oren1982}
I.~Oren: {\em Admissible functions with multiple discontinuities},
Israel J. Math. 42 (1982),  no. 4, 353--360.


\vspace*{.05in}

\bibitem{Ostr1927/30}
A.~Ostrowski: Math. Miszelen IX and XVI, {\em Notiz zur theorie der Diophantischen approximationen}, Jahresber. d. Deutschen Math. Ver. 36 (1927), 178-180 and 39 (1930), 34--46.


\vspace*{.05in}

\bibitem{Rauz1972}
G.~Rauzy: {\em Ensembles \`a restes born\'es}, Seminaire de Th\'eorie des Nombres de Bordeaux (1972), 1--12.

\vspace*{.05in}




\bibitem{RamaVale1999}
D.~Ramakrishnan, R.~J.~Valenza:
{\em Fourier analysis on number fields},
Graduate Texts in Mathematics 186, Springer-Verlag, New York,  1999.


\vspace*{.05in}



\bibitem{Szus1954}
P.~Sz\"{u}sz: {\em \"{U}ber die Verteilung der Vielfachen einer komplexen Zahl nach dem
	Modul des Einheitsquadrats} (German), Acta Math. Acad. Sci. Hungar.  5  (1954), 35--39.


\vspace*{.05in}


\bibitem{Zhur2005}
V.~G.~Zhuravlev: {\em Rauzy tilings and bounded remainder sets on a torus}
(Russian),  Zap. Nauchn. Sem. S.-Peterburg. Otdel. Mat. Inst. Steklov. (POMI)  322  (2005),  Trudy po Teorii Chisel, 83-106, 253; translation in  J. Math. Sci. (N. Y.)  137  (2006),  no. 2, 4658--4672



\vspace*{.05in}

\bibitem{Zhur2011}
V.~G.~Zhuravlev: {\em Exchanged toric developments and bounded remainder sets} (Russian),  Zap. Nauchn. Sem. S.-Peterburg. Otdel. Mat. Inst. Steklov. (POMI)  392  (2011),  Analiticheskaya Teoriya Chisel i Teoriya Funktsii. 26, 95--145, 219--220; translation in  J. Math. Sci. (N.Y.)  184  (2012),  no. 6, 716--745.



\vspace*{.05in}


\bibitem{Zhur2012}
V.~G.~Zhuravlev: {\em Bounded remainder polyhedra}, Sovremennye Problemy Matematiki, 2012, Vol. 16, pp. 82--102.



\end{thebibliography}
\end{document}